\documentclass[12pt]{amsart}

\usepackage{amssymb}

\usepackage[colorlinks = true, citecolor = blue]{hyperref} 
\newcommand{\C}{\mathbb{C} }
\newcommand{\R}{\mathbb{R} }
\newcommand{\Z}{\mathbb{Z} }
\newcommand{\T}{\mathbb{T} }

\newcommand{\N}{\mathbb{N} }

\newcommand{\indicator}[1]{{\bf 1}_{#1}}
\newcommand{\eof}[1]{e \inparentheses{#1}}

\newcommand{\inparentheses}[1]{\left( #1 \right)}

\renewcommand{\vector}[1]{{\bf #1}}
\newcommand{\form}{\mathcal{Q}}
\newcommand{\kernel}{K}

\DeclareMathOperator{\dilate}{D}
\DeclareMathOperator{\Kloosterman}{Kl}
\DeclareMathOperator{\high}{high}
\DeclareMathOperator{\low}{low}

\newtheorem{proposition}{Proposition}
\newtheorem{lemma}{Lemma}
\newtheorem{theorem}{Theorem}
\newtheorem{corollary}{Corollary}
\newtheorem{conjecture}{Conjecture}
\newtheorem{question}{Question}

\newtheorem*{Littman}{Littman}
\newtheorem{Magyar}{Magyar}

\theoremstyle{definition}

\newtheorem{example}{Example}

\theoremstyle{remark}
\newtheorem{remark}{Remark}[section]

\title{$\ell^p$-improving for discrete spherical averages}
\author{Kevin Hughes}
\address{
    School of Mathematics
	\\	The University of Bristol
	\\	Howard House
	\\	Queens Avenue
	\\	Bristol, BS8 1TW
	\\	UK
	\\ and the Heilbronn Insitute for Mathematical Research, Bristol, UK
}
\email{khughes.math@gmail.com}

\begin{document}
\maketitle
%
%

\begin{abstract}
We initiate the theory of $\ell^p$-improving inequalities for arithmetic averages over hypersurfaces and their maximal functions. 
In particular, we prove $\ell^p$-improving estimates for the discrete spherical averages and some of their generalizations. 
As an application of our $\ell^p$-improving inequalities for the dyadic discrete spherical maximal function, we give a new estimate for the full discrete spherical maximal function in four dimensions. 
Our proofs are analogous to Littman's result on Euclidean spherical averages. 
One key aspect of our proof is a Littlewood--Paley decomposition in both the arithmetic and analytic aspects. 
In the arithmetic aspect this is a major arc-minor arc decomposition of the circle method. 
\end{abstract}

\section{Introduction}

The motivation for this paper is Littman's $L^p(\R^d)$-improving result for spherical averages from \cite{Littman}. 
For dimensions $d \geq 2$ and functions $f : \R^d \to \C$ define the spherical average (over the unit sphere) by 
\begin{equation*}
\mathcal{A}f(\vector{x}) := \int_{\mathbb{S}^{d-1}} f(\vector{x}-\vector{y}) \; d\sigma(\vector{y})
\end{equation*}
where $d\sigma$ is the Euclidean surface measure on the unit sphere $\mathbb{S}^{d-1}$ in $\R^d$. 
\begin{Littman}
If $\mathcal{A}$ is the averaging operator over the unit sphere, then 
\begin{equation*}\label{eq:Littman}
\|\mathcal{A}f\|_{L^{d+1}(\R^d)} \leq C \|f\|_{L^{\frac{d+1}{d}}(\R^d)}.
\end{equation*}
\end{Littman}
\noindent 
In this note we will be interested in estimates for the discrete spherical averages which are analogous to \eqref{eq:Littman}. 
Suppose that $d \geq 2$.
For $\lambda \in \N$ and functions $f : \Z^d \to \C$, define the discrete spherical averages
\[
A_\lambda f(\vector{x}) := N_d(\lambda)^{-1} \sum_{\vector{y} \in \Z^d : |\vector{y}|^2=\lambda} f(\vector{x}-\vector{y})
\]
whenever $N_d(\lambda) := \#\{ \vector{y} \in \Z^d : |\vector{y}|^2=\lambda \}$ is not zero. 
In other words, $A_\lambda$ is the linear operator given by convolution with the discrete (or more appropriately named ``arithmetic'') probability measure 
\[
\sigma_\lambda := N_d(\lambda)^{-1} {\bf 1}_{\{ \vector{y} \in \Z^d : |\vector{y}|^2=\lambda \}}.
\]

Our main result is the following. 
\begin{theorem}\label{theorem:discrete_improving}
If $d \geq 4$ and $\frac{d+1}{d-1} \leq p \leq 2$, then for each $\epsilon>0$, there exists constants $C_{p,\epsilon}$ depending on $p$ and $\epsilon$ such that for all $\lambda \in \N$ (further restrict $\lambda$ to be odd when $d=4$), we have the $\ell^p$-improving inequality
\begin{equation}\label{eq:discrete_improving}
\|A_\lambda f\|_{\ell^{p'}(\Z^d)} 
\leq C_{p,\epsilon} 
\lambda^{\epsilon-\frac{d}{2}(\frac{2}{p}-1)} \|f\|_{\ell^{p}(\Z^d)}.
\end{equation}
The implicit constants are independent of $\lambda$. 
\end{theorem}
\noindent 
We are also motivated by Lee's work \cite{Lee} which proved that the dyadic spherical maximal function variant of Littman's theorem holds. 
Moreover, our methods are flexible and we can use them to strengthen Theorem~\ref{theorem:discrete_improving} when $p$ is sufficiently large. 
\begin{theorem}\label{theorem:improving:dyadic}
If $d \geq 5$ and $p := \frac{d}{d-2}$, then there exists constants $C_{p}$ depending on $p$ such that for all $\Lambda \in \N$, we have the $\ell^p$-improving inequality for the $\Lambda$-dyadic discrete spherical maximal function
\begin{equation}\label{eq:discrete_improving}
\| \sup_{\Lambda \leq \lambda < 2\Lambda} |A_\lambda f| \|_{\ell^{p',\infty}(\Z^d)} 
\leq C_{p} 
\Lambda^{-\frac{d}{2}(\frac{2}{p}-1)} \|f\|_{\ell^{p,1}(\Z^d)}.
\end{equation}
The constant $C_p$ is independent of $\Lambda$. 
\end{theorem}
\noindent 
This clearly implies the same is true for a single average and hence we may remove the $\epsilon$-loss in Theorem~\ref{theorem:discrete_improving} when $p > \frac{d}{d-2}$.


\subsection{Motivation}

When $d \geq 5$ and $\lambda \in \N$, we have that $100^{-1} \lambda^{\frac{d-2}{2}} \leq N_d(\lambda) \leq 100 \lambda^{\frac{d-2}{2}}$; when $d=4$ and $\lambda$ is restricted to be odd, we have that $100^{-1} \lambda^{\frac{d-2}{2}}/\log \log \lambda \leq N_d(\lambda) \leq 100 \lambda^{\frac{d-2}{2}} \cdot \log \log \lambda$. 
Several years ago Jim Wright asked the author: \emph{What is $\ell^p$-improving for the discrete spherical averages?}
We interpret his question as the following. 
\begin{question}[Jim Wright]\label{question:improving}
When are there exponents $1 \leq p,q \leq \infty$ and a constant $C = C_{d,p,q}$, possibly depending on $d,p,q$ but independent of $\lambda$ such that 
\begin{equation}\label{eq:improving}
\|A_\lambda f\|_{\ell^{q}(\Z^d)} \leq C \|f\|_{\ell^{p}(\Z^d)}?
\end{equation}
\end{question}

On the one hand, unlike the continuous case in $\R^d$, we do not have the dilational symmetry to exploit; this is why we want \eqref{eq:improving} to hold uniformly for $\lambda \in \N$ instead of formulating the question for the unit sphere ($\lambda=1$) as stated in Littman's result. 
On the other hand, we may quickly obtain some trivial off-diagonal results by using the contraction inequality
\begin{equation}\label{eq:contraction}
\|A_\lambda f\|_{\ell^{p}(\Z^d)} \leq \|f\|_{\ell^{p}(\Z^d)} 
\quad \text{for all} \quad 1 \leq p \leq \infty
\end{equation}
and the nesting property of $\ell^p$-spaces
\begin{equation}\label{eq:nesting}
\|f\|_{\ell^q} \leq \|f\|_{\ell^p}
\quad \text{for all} \quad
1 \leq p \leq q \leq \infty,
\end{equation}
to see that \eqref{eq:improving} is true for all $\lambda$ for $1 \leq p \leq q \leq \infty$ with $C=1$. 
This makes Question~\ref{question:improving} trivial. 

We consider the following two examples; assume for simplicity that $d \geq 5$ and $\lambda \in \N$ is large. 
\begin{example}[$\delta$-function]\label{example:delta}
Take $f$ to be the delta function at the origin. 
Then $A_\lambda f$ is supported on the sphere of radius $\sqrt{\lambda}$ and has height $ N_d(\lambda)^{-1} \eqsim \lambda^{\frac{2-d}{2}}$. 
Thus, for $p,q \geq 1$, 
\[
\| A_\lambda f \|_{\ell^q(\Z^d)} 
= 
\| A_\lambda f \|_{\ell^\infty(\Z^d)} N_d(\lambda)^{\frac{1}{q}} 
=
N_d(\lambda)^{\frac{1}{q}-1} 
= 
N_d(\lambda)^{\frac{1}{q}-1} \|f\|_{\ell^p(\Z^d)} 
\eqsim 
\lambda^{-(\frac{d-2}{2})(1-\frac{1}{q})} \|f\|_{\ell^p(\Z^d)}.
\]
We have a similar bound for the dyadic maximal function
\[
\| \sup_{\Lambda \leq \lambda < 2\Lambda} |A_\lambda f| \|_{\ell^q(\Z^d)}^q 
= 
\sum_{\Lambda \leq \lambda < 2\Lambda} \sum_{\vector{x} \in \Z^d : |\vector{x}| = \lambda}  |A_\lambda f(\vector{x})|^q.
\]
Consequently, 
\[
\| \sup_{\Lambda \leq \lambda < 2\Lambda} |A_\lambda f| \|_{\ell^q(\Z^d)}^q 
\lesssim 
\Lambda \cdot \Lambda^{\frac{d-2}{2}} \cdot \Lambda^{-q(\frac{d-2}{2})} 
= 
\Lambda^{1+ (1-q)(\frac{d-2}{2})}. 
\]
Therefore, for $p,q \geq 1$, we have 
\[
\| \sup_{\Lambda \leq \lambda < 2\Lambda} |A_\lambda f| \|_{\ell^q(\Z^d)}
\lesssim 
\Lambda^{\frac{1}{q} + (\frac{1}{q}-1)(\frac{d-2}{2})} 
= 
\Lambda^{\frac{1}{q} + (\frac{1}{q}-1)(\frac{d-2}{2})} \|f\|_{\ell^p(\Z^d)}.
\]
\end{example}

\begin{example}[big ball-function]\label{example:bigball}
Take $f$ to be the indicator function of a ball of radius $R \eqsim \sqrt{\lambda}$. 
Then $A_\lambda f$ is supported on the ball of radius $R+ \sqrt{\lambda} \eqsim \sqrt{\lambda}$ and has height $\gtrsim 1$ for a large chunk of it. 
Therefore,
\[
\| A_\lambda f \|_{\ell^q(\Z^d)} \eqsim  \|f\|_{\ell^q(\Z^d)} \eqsim R^{d/q}
\]
for $q \geq 1$ while 
\[
\|f\|_{\ell^p(\Z^d)} \eqsim R^{d/p}.
\]
The same estimates hold with the dyadic maximal function $\sup_{\Lambda \leq \lambda < 2\Lambda} |A_\lambda f|$ in place of $A_\lambda f$. 
Here we immediately see that we must have {$q \geq p$ to satisfy \eqref{eq:improving}}. 
Combining these two we see that for $1 \leq p \leq 2$:
\[
\| \sup_{\Lambda \leq \lambda < 2\Lambda} |A_\lambda f| \|_{\ell^{p'}(\Z^d)} \eqsim  \Lambda^{-\frac{d}{2}(\frac{2}{p}-1)} \|f\|_{\ell^p(\Z^d)}.
\]
\end{example}

Here and throughout, $p'$ is the dual exponent to $p$ which means for $1 \leq p \leq \infty$ that $\frac{1}{p} + \frac{1}{p'} = 1$. 

Example~\ref{example:delta} actually reveals more since Young's inequality implies
\[
\|A_\lambda f\|_{\ell^\infty(\Z^d)} 
= 
\|\sigma_\lambda * f\|_{\ell^\infty(\Z^d)} 
\leq 
\|\sigma_\lambda\|_{\ell^\infty(\Z^d)} \|f\|_{\ell^1(\Z^d)} 
\leq 
100 \lambda^{-\frac{d-2}{2}} \|f\|_{\ell^1(\Z^d)}. 
\]
This bound extends to the dyadic maximal functions:
\[
\| \sup_{\Lambda \leq \lambda < 2\Lambda} |A_\lambda f| \|_{\ell^\infty(\Z^d)} 
\leq 
100 \Lambda^{-\frac{d-2}{2}} \|f\|_{\ell^1(\Z^d)}. 
\]

Interpolating this with the contraction inequality we obtain the following estimate
\begin{equation}\label{eq:trivial_bound}
\| \sup_{\Lambda \leq \lambda < 2\Lambda} |A_\lambda f| \|_{\ell^{p'}(\Z^d)} 
\leq 
100\Lambda^{-\frac{d-2}{2}(\frac{2}{p}-1)} \|f\|_{\ell^p(\Z^d)} 
\quad \text{for} \quad 1 \leq p \leq 2,
\end{equation}
which we call the \emph{trivial bound}.
For large $\lambda$ the trivial bound is larger than the bounds from our examples. 
%
Curiously, the trivial bound shows that there exist $\ell^p(\Z^d)$-improving estimates which decay with $\lambda$. 
Therefore a better interpretation of Jim Wright's question is the following.
\begin{question}\label{question:sharp}
For each $1<p<2$, what is the best exponent $\eta_p$ so that 
\begin{equation*}
\|A_\lambda f\|_{\ell^{p',\infty}(\Z^d)} 
\leq C_{p} 
\lambda^{-\eta_p} \|f\|_{\ell^{p}(\Z^d)}
\end{equation*}
where the constant $C_p$ is independent of $\lambda \in \N$ and $f \in \ell^p(\Z^d)$?
\end{question}

We also ask this question for the dyadic version. 
\begin{question}
For each $1<p<2$, what is the best exponent $\nu_p$ so that 
\begin{equation*}
\| \sup_{\Lambda \leq \lambda < 2\Lambda} |A_\lambda f| \|_{\ell^{p',\infty}(\Z^d)} 
\leq C_{p}' 
\Lambda^{-\nu_p} \|f\|_{\ell^{p}(\Z^d)}
\end{equation*}
where the constant $C_p'$ is independent of $\Lambda \in \N$ and $f \in \ell^p(\Z^d)$?
\end{question}

Theorem~\ref{theorem:discrete_improving} addresses Question~\ref{question:sharp} when we are close to $p=2$. 
In particular it says that in Question~\ref{question:sharp} we may take any $0 \leq \eta_p<\frac{d}{2}(\frac{2}{p}-1)$  for $\frac{d+1}{d-1} \leq p <2$. 
Comparing \eqref{eq:discrete_improving} to \eqref{eq:trivial_bound}, we see that Theorem~\ref{theorem:discrete_improving} improves upon the trivial bound. 
Recall that the trivial bound was given by a simple convexity estimate (in this case applying Young's inequality); consequently, Theorem~\ref{theorem:discrete_improving} might be referred to as a \emph{subconvexity estimate} in the argot of analytic number theorists. 

One can also ask for $\ell^p$-improving estimates for the full discrete spherical maximal function. 
For $f : \Z^d \to \C$, let 
\[
A_* f(\vector{x}) := \sup_{\lambda \in \N : \lambda \; \text{is odd}} |A_\lambda f(\vector{x})|
\]
denote the discrete spherical maximal function when $d=4$, 
and 
\[
A_* f(\vector{x}) := \sup_{\lambda \in \N} |A_\lambda f(\vector{x})|
\]
denote the discrete spherical maximal function when $d \geq 5$. 

\begin{question}
When are there exponents $1 \leq p,q \leq \infty$ so that 
\begin{equation}\label{eq:improving:maximal}
\|A_* f\|_{\ell^{q}(\Z^d)} \lesssim \|f\|_{\ell^{p}(\Z^d)}?
\end{equation}
\end{question}
When $d \geq 5$, there are some obvious ranges for which one may obtain bounds. 
In particular, by applying Magyar--Stein--Wainger's theorem on the discrete spherical maximal function, see "Theorem" in \cite{MSW}, we have that \eqref{eq:improving:maximal} holds for $d \geq 5$ and $\frac{d}{d-2} <q \leq \infty$ because
\[
\|A_*f\|_{\ell^q} \lesssim \|f\|_{\ell^q}
\]
for all $q > \frac{d}{d-2}$ and the nesting property of $\ell^p$-spaces implies that 
\[
\|A_*f\|_{\ell^{q}} \lesssim \|f\|_{\ell^{q}} \leq \|f\|_{\ell^p}
\]
for $1 \leq p \leq q$. 
Moreover this range is sharp; this can be seen by considering the delta-function example. 
In four dimensions, the analogous  Magyar--Stein--Wainger estimate is expected to hold. 
\begin{conjecture}
If $p>2$, then 
\begin{equation}
\|A_* f\|_{\ell^{p}(\Z^4)} 
\lesssim_{p} 
\|f\|_{\ell^{p}(\Z^4)}.
\end{equation}
\end{conjecture}

Interestingly we can prove the following $\ell^p$-improving result when $d=4$ which would be a corollary of Conjecture~1 by the nesting properties of $\ell^p$-spaces. 
\begin{theorem}\label{theorem:4d}
Assume that $d=4$. 
If $q>2$ and $1 \leq p < q$, then
\begin{equation}\label{eq:maximal}
\|A_* f\|_{\ell^{q}(\Z^4)} 
\lesssim_{p,q} 
\|f\|_{\ell^{p}(\Z^4)}.
\end{equation}
\end{theorem}

\subsection{Comparison to recent works}

Our examples show that \eqref{eq:discrete_improving} fails for $p<\frac{d+2}{d}$. 
So, one might expect that \eqref{eq:discrete_improving} would hold for all $\frac{d+2}{d} \leq p \leq 2$ which would go beyond Theorem~\ref{theorem:discrete_improving}. 
Intriguingly, Kesler--Lacey \cite{Kesler} showed that \eqref{eq:discrete_improving} fails for $p<\frac{d+1}{d-1}$. 
Moreover \cite{Kesler} removed the $\epsilon$-loss in \eqref{eq:discrete_improving} for $\frac{d+1}{d-1} < p < \frac{d}{d-2}$. 

We encourage the reader to read Kesler--Lacey's interesting work \cite{Kesler} which appeared independently of this paper. 
Kesler--Lacey also considered $\ell^p(\Z^d)$-improving inequalities for discrete spherical averages. 
In \cite{Kesler} their focus lied upon using $\ell^p$-improving inequalities to deduce `sparse bounds' for the discrete spherical maximal function. 
The question of sparse bounds is not considered here. 
However, we prove $\ell^p$-improving inequalities for a broader class of averages which are not considered in \cite{Kesler}.

As the reader may compare, the ideas in \cite{Kesler} - for proving $\ell^p(\Z^d)$-improving inequalities for discrete spherical averages - are very similar to those here. 
Their methods may appear more complicated due to the use of the Ramanujan bound from \cite{Bourgain, Hughes_sparse}. 
We point out that the Ramanujan bound first appeared in Bourgain's work on restriction of the parabola to the integer lattice to remove an $\epsilon$-loss there, and the Ramanujan bound was first adapted to the context of spherical averages in my work on discrete spherical maximal functions over sparse subsequences of radii; see \cite{Hughes_sparse}. 

\subsection{Overview of the proofs}

For many discrete analogues in harmonic analysis not only are the statements of theorems analogous, but there is also an analogy between their proofs. 
We take a moment to describe this since this analogy does not appear to be explained in the literature. 

For many problems in Fourier analysis such as Littman's theorem, one decomposes an operator into `low' and `high' frequencies and obtains bounds for these different pieces. 
For instance the Littlewood--Paley square function is one way to do this. 
There is a similar decomposition for analogous problems with an arithmetic flavor. 

For problems over $\Z^d$ instead of $\R^d$, all frequencies in the torus $(\R/\Z)^d$ are ostensibly low frequencies since the torus is compact. 
Unfortunately, this perspective is insufficiently nuanced to treat many problems. 
Instead one should recalibrate to the following: replace the sobriquets `low frequencies' and `high frequencies' of Fourier analysis with `major arcs' and `minor arcs' of the circle method respectively. 

Recalibrating one's perspective to the above analogy allows us to import intuitions and paradigms from continuous Euclidean harmonic analysis to discrete Euclidean harmonic analysis via the circle method. 
By way of this analogy one sees that the circle method is akin to Littlewood--Paley theory. 
I encourage the reader to review the proofs of Littman's theorem before reading the proofs here to see this analogy in action. 


\subsection{Organization of the paper}

The paper is organized as follows. 
Section~\ref{section:notation} sets some notation used throughout the paper. 
Section~\ref{section:BirchMagyar} generalizes Theorem~1 to hypersurfaces defined by nice, positive definite homogeneous forms with integral coefficients. 
We prove Theorem~\ref{theorem:discrete_improving} in Section~\ref{section:sphere} by making use of improved estimates for Kloosterman sums. 
Section~\ref{section:sphere:dyadic} proves bounds for dyadic discrete $k$-spherical maximal functions. 
In Section~\ref{section:4d} we prove Theorem~\ref{theorem:4d}. 
Section~\ref{section:conclusion} concludes with a few questions. 
Finally, in the Appendix we record the range of $\ell^p$-boundedness of Magyar's maximal functions (which arise in Section~\ref{section:BirchMagyar}) from \cite{Magyar:ergodic}.

\section{Notation}\label{section:notation}

We introduce here some notation that will streamline our exposition. 
\begin{itemize}
\item
We write $f(\lambda) \lesssim g(\lambda)$ if there exists a constant $C>0$ independent of all $\lambda$ under consideration (e.g. $\lambda$ in $\N$ or in $\Gamma_{\form}$) such that 
\[
|f(\lambda)| \leq C |g(\lambda)|. 
\]
Furthermore, we will write $f(\lambda) \gtrsim g(\lambda)$ if $q(\lambda) \lesssim f(\lambda)$ while we will write $f(\lambda) \eqsim g(\lambda)$ if $f(\lambda) \lesssim g(\lambda)$ and $f(\lambda) \gtrsim g(\lambda)$
\item
Subscripts in the above notations will denote parameters, such as the dimension $d$ or degree $k$ of a form $\form$, on which the implicit constants may depend. 
\item
$\T^d$ denotes the $d$-dimensional torus $(\R/\Z)^d$ identified with the unit cube $[-1/2,1/2]^d$. 
\item
$*$ denotes convolution on a group such as $\Z^d$, $\T^d$ or $\R^d$. 
It will be clear from context as to which group the convolution takes place.
\item
$\eof{t}$ will denote the character $e^{-2\pi it}$ for $t \in \R$ or $\T$
\item
For a function $f: \Z^d \to \C$, its $\Z^d$-Fourier transform will be denoted $\widehat{f}(\xi)$ for $\xi \in \T^d$. 
For a function $f: \R^d \to \C$, its $\R^d$-Fourier transform will be denoted $\widetilde{f}(\xi)$ for $\xi \in \R^d$. 
\item
For a function $f: \R^d \to \C$, we define dilation operator $\dilate_{t}$ by $\dilate_{t}f(\vector{x}) = f(\vector{x}/t)$. 
\item
For a ring $R$, we will use the inner product notation $\vector{b} \cdot \vector{m}$ for vectors $\vector{b},\vector{m} \in R^d$ to mean the sum $\sum_{i=1}^d b_i m_i$. 
This is used for the rings $\R^,\Z,\T$ and $\Z/q$ where $q \in \Z$. 
\item
We also let ${\bf 1}_X$ denote the indicator function of the set $X$. 
\end{itemize}

\section{$\ell^p$-improving for Magyar's theorem}\label{section:BirchMagyar}

Our method is quite general, so we start by generalizing Theorem~\ref{theorem:improving:dyadic}. 
Throughout this section $\form(\vector{x}) \in \Z[x_1,x_2,\dots,x_d]$, where $\vector{x} = (x_1,x_2,\dots,x_d)$, will denote an integral, positive definite, homogeneous form, and $k$ will denote the degree of the form $\form$. 
We assume that $k \geq 2$ and a natural number. 
Let $V_\form(\C) := \{\vector{x} \in \C^d :  \nabla \form(\vector{x}) = \vector{0}\}$ denote the (Birch) \emph{singular locus} of the form $\form$. 
We will say that a homogeneous, integral form is \emph{non-singular} if it satisfies Birch's criterion:
\begin{equation}\label{Birch_criterion}
d-\dim_\C(V_\form(\C)) > (k-1)2^k.
\end{equation}
The notion of dimension `$\dim_\C(V_\form(\C))$' can be taken to be the algebraic dimension of the complex variety $V_\form(\C)$.

When \eqref{Birch_criterion} is satisfied, Birch \cite{Birch} tells us that there exists a positive constant $C_\form$ and an infinite arithmetic progression $\Gamma_{\form}$ in $\N$ depending on the form $\form$ so that 
\[
N_{\form}(\lambda) 
:= \#\{ \vector{n} \in \Z^d : \form(\vector{n}) = \lambda \} 
\geq C_{\form} \lambda^{\frac{d}{k}-1}
> 0
\quad \text{for all} \quad \lambda \in \Gamma_{\form}.
\]
Following Magyar \cite{Magyar:ergodic}, we will call any such arithmetic progression $\Gamma_\form$ a \emph{set of regular values for $\form$}. 
For each $\lambda \in \N$, $N_{\form}(\lambda)$ is finite because the hypersurface $\{ \vector{n} \in \R^d : \form(\vector{n}) = \lambda \} $ is defined by a positive definite form which implies that this hypersurface is compact. 
Consequently, the averages 
\[
A^\form_\lambda(\vector{x}) := N_{\form}(\lambda)^{-1} \sum_{\vector{n} \in \Z^d : \form(\vector{n})=\lambda} f(\vector{x}-\vector{n})
\]
make sense for all $\lambda \in \Gamma_{\form}$ and functions $f:\Z^d \to \C$. 
In this setting, our trivial bound \eqref{eq:trivial_bound} becomes
\begin{equation}\label{eq:Birch:trivial_bound}
\| \sup_{\Lambda \leq \lambda < 2\Lambda} |A^\form_\lambda f| \|_{\ell^{p'}(\Z^d)} 
\lesssim 
\Lambda^{-(\frac{d}{k}-1)(\frac{2}{p}-1)} \|f\|_{\ell^p(\Z^d)} 
\quad \text{for} \quad 1 \leq p \leq 2
\quad \text{and} \quad \Lambda \geq 1.
\end{equation}
This follows from Young's inequality and Birch's estimate for $N_\form(\lambda)$. 
In all of our dyadic maximal functions we will restrict to $\lambda \in \Gamma_\form$. 

For a non-singular, homogeneous, integral form define the parameters
\begin{equation}\label{eq:Birch:parameters}
\gamma_{\form} := \frac{1}{6k} \left( \frac{d-\dim(V_\form(\C))}{(k-1)2^k} - 1 \right)
\quad \text{and} \quad
\kappa_{\form} := \frac{d-\dim V_\form(\C)}{2^{k-1}(k-1)}.
\end{equation}
Throughout we assume that $d > k \geq2$ with $d$ sufficiently large with respect to $k$ to satisfy the Birch--Magyar non-singularity criterion \eqref{Birch_criterion} so that $\gamma_\form>0$ and $\kappa_\form>2$.
The following result gives an improvement over the trivial bound \eqref{eq:Birch:trivial_bound} when $p$ is close to 2.
\begin{theorem}\label{theorem:discrete_improving:Birch}
Let $\form$ be a positive definite, non-singular, homogeneous, integral form in $d$ variables of degree $k$ and $\Gamma_\form$ a set of regular values for $\form$. 
If $p := \frac{\kappa}{\kappa-1}$, then for each $\epsilon>0$ there exists a constant $C_{\epsilon,\form,p}$ independent of $\Lambda \geq 1$ so that 
\begin{equation}\label{eq:discrete_improving:Birch}
\|\sup_{\Lambda \leq \lambda < 2\Lambda : \lambda \in \Gamma_\form} |A^\form_\lambda f| \|_{\ell^{p'}(\Z^d)} 
\leq 
C_{\form,p,\epsilon} \Lambda^{\epsilon-\frac{d}{k}(\frac{2}{p}-1)} 
\left( 1 + \Lambda^{\frac{2}{p}-1 - \gamma_\form(2-\frac{2}{p})} \right)
\|f\|_{\ell^{p}(\Z^d)}
\end{equation}
for all $\Lambda \geq 1$. 
Recall that the supremum is restricted so that each $\lambda$ is in $\Gamma_\form$. 
\end{theorem}

Note that our assumption that $\kappa > 2$ implies that $\kappa/(\kappa-1) < 2$. 
Our theorem implies the following corollary which says that our dyadic maximal functions satisify essentially sharp $\ell^p(\Z^d) \to \ell^{p'}(\Z^d)$-improving estimates for $p \leq2$ with $p$ sufficiently close to 2. 
\begin{corollary}
Let $\form$ be a positive definite, non-singular, homogeneous, integral form in $d$ variables of degree $k$ and $\Gamma_\form$ a set of regular values for $\form$. 
If $\frac{2(1+\gamma_\form)}{1+2\gamma_\form} \leq p \leq 2$, then for each $\epsilon>0$ there exists a constant $C_{\epsilon,\form,p}$ independent of $\Lambda \geq 1$ so that 
\[
\|\sup_{\Lambda \leq \lambda < 2\Lambda : \lambda \in \Gamma_\form} |A^\form_\lambda f| \|_{\ell^{p'}(\Z^d)} 
\leq 
C_{\form,p,\epsilon} \Lambda^{\epsilon-\frac{d}{k}(\frac{2}{p}-1)} \|f\|_{\ell^{p}(\Z^d)}
\]
for all $\Lambda \geq 1$. 
\end{corollary}
Our corollary follows by determining when $\frac{2}{p}-1 - \gamma_\form(2-\frac{2}{p}) \leq 0$ and noting that $\gamma_\form < \kappa_\form$. 
This is a computation that we leave to the reader. 

The heavy lifting in our thoerem lies in a decomposition of Magyar for the averages $A^\form_\lambda$. 
\begin{Magyar}[\cite{Magyar:distribution}]\label{Magyar:Birch}
Let $\form(\vector{x}) \in \Z[\vector{x}]$ be a positive definite, non-singular, integral, homogeneous form, and $\Gamma_{\form}$ be a set of regular values for the form $\form$. 
For each $\lambda \in \Gamma_\form$ the averaging operator $A^\form_\lambda$ decomposes into the sum of two convolution operators, $A_\lambda = M^\form_\lambda+E^\form_\lambda$ such that 
for all $\Lambda \geq 1$, we have  
\begin{equation}\label{eq:error_bound}
\| \sup_{\Lambda \leq \lambda < 2\Lambda : \lambda \in \Gamma_\form} |E_\lambda| \|_{\ell^2(\Z^d) \to \ell^2(\Z^d)} 
\lesssim_{\form,\gamma} 
\Lambda^{\epsilon-\gamma_\form} 
\quad \text{for all} \quad \epsilon>0
\end{equation}
The implicit constant is independent of $\Lambda \geq 1$. 
The main term, $M^\form_\lambda$, is the sum of finitely many convolution operators, 
\[
M^\form_\lambda 
= 
\sum_{q=1}^{\Lambda^{1/k}} \sum_{a \in (\Z/q)^*} \eof{a\lambda/q} M^{\form,a/q}_\lambda,
\] where $\Lambda = 2^j$ satisfies $\Lambda/2 \leq \lambda < \Lambda$. 
Moreover, $M^\form_\lambda$ satisfies the estimate  
\begin{equation}\label{eq:Weyl_bound}
\| \sup_{\lambda \in \Gamma_\form} |M^{\form,a/q}_\lambda| \|_{\ell^2(\Z^d) \to \ell^2(\Z^d)} 
\lesssim_{\epsilon} 
q^{\epsilon-\kappa_\form} 
\quad \text{for all} \quad \epsilon>0.
\end{equation}
\end{Magyar}

We remark that $M^{\form,a/q}_\lambda$ is the convolution operator corresponding to the Fourier multiplier
\begin{equation*}
\widehat{M^{\form,a/q}_\lambda}(\vector{\xi}) 
:= 
\sum_{\vector{m} \in \Z^d} G_{\form}(a,q;\vector{m}) \Psi(q\vector{\xi}-\vector{m}) \widetilde{d\sigma_{\form}}(\lambda^{1/k}[\vector{\xi}-\frac{\vector{m}}{q}]).
\end{equation*}
For $\vector{m} \in \Z^d$ 
\[
G_\form(a,q;\vector{m}) 
:= 
q^{-d}\sum_{\vector{b} \in (\Z/q)^d} \eof{\frac{a\form(\vector{b})+\vector{b}\cdot\vector{m}}{q}}
\]
is the normalized Weyl sum satisfying the bound 
\begin{equation}\label{eq:Birch:Weyl_sum}
|G_\form(a,q;\vector{m})| \lesssim_\epsilon q^{\epsilon-\kappa}
\quad \text{for all} \quad \epsilon>0
\end{equation}
uniformly in $a \in (\Z/q)^*$ and $\vector{m} \in \Z^d$. 
The function $\Psi$ is a $C^\infty(\R^d)$ bump function supported in the cube $[-1/4,1/4]^d$ and 1 on the cube $[-1/8,1/8]^d$, and the singular measure $d\sigma_{\form}$ is the Gelfand--Leray form defined distributionally by the oscillatory integral 
\[
\int_{\R} \eof{t(\form(\vector{x})-1)} \; dt.
\]
Alternatively, $d\sigma_{\form}(\vector{x}) = dS_{\form}(\vector{x})/|\nabla \form(\vector{x})|$ where $dS_{\form}$ is the Euclidean surface area measure on the hypersurface $\{ \vector{x} \in \R^d : \form(\vector{x})=1 \}$ which is compactly supported since $\form$ is positive definite. 
See \cite{Magyar:ergodic} for more information concerning Gelfand--Leray measures of hypersurfaces. 
We cite the following bound - see Lemma~6 on page 931 of \cite{Magyar:ergodic} - for the $\R^d$-Fourier transform of the surface measure: 
\begin{equation}\label{eq:Birch:surface_measure}
|\widetilde{d\sigma_\form}(\vector{\xi})| 
\lesssim_\epsilon 
(1+|\vector{\xi}|)^{1-\kappa+\epsilon} 
\quad \text{for each} \quad \vector{\xi} \in \R^d
\quad \text{and for all} \quad \epsilon>0.
\end{equation}

The estimate \eqref{eq:error_bound} does not explicitly appear in \cite{Magyar:ergodic}. 
Instead it appears for a slightly different definition of our error term, so we briefly indicate how one obtains it. 
Estimate \eqref{eq:error_bound} is encoded in the proofs of Propositions~3 and 4 in \cite{Magyar:ergodic}. 
One key difference in this paper is that our main term $M^\form_\lambda$ is a finite sum depending on $\lambda$, and so we do not need (2.17) of Proposition~4 in \cite{Magyar:ergodic}. 
Meanwhile the estimates (2.15) and (2.16) of Proposition~4 in \cite{Magyar:ergodic} are superior to the minor arc estimate of Proposition~3 in \cite{Magyar:ergodic}. 
Therefore the minimal exponent which defines $\gamma_\form$ comes from the minor arc estimate.

Magyar's theorem gives us $\ell^2$-estimates, but we are interested in $\ell^p \to \ell^{p'}$-estimates. 
We will interpolate Magyar's $\ell^2$-estimates with appropriate $\ell^1 \to \ell^\infty$-estimates to deduce the following lemmas which when added together immediately yield Theorem~\ref{theorem:discrete_improving:Birch}. 
Our lemma for the main term is the following. 
\begin{lemma}\label{lemma:main:lp}
Let $\form(\vector{x}) \in \Z[\vector{x}]$ be a positive definite, non-singular, integral, homogeneous form satisfying \eqref{Birch_criterion}, and $\Gamma_{\form}$ be a set of regular values for the form $\form$. 
If $p := \frac{\kappa}{\kappa-1}$ (which is less than 2 by our assumption on $\kappa$), then 
\begin{equation}\label{eq:main_term} 
\| \sup_{\Lambda \leq \lambda < 2\Lambda} |M_\lambda| \|_{\ell^{p,1}(\Z^d) \to \ell^{p',\infty}(\Z^d)} 
\lesssim 
\Lambda^{\epsilon-\frac{d}{k}(\frac{2}{p}-1)} 
\quad \text{for all} \quad \Lambda \geq 1.
\end{equation}
\end{lemma}
\noindent 
Up to a factor of $\Lambda^\epsilon$ the bound for our main term is the size we expect for our averages in the range $\frac{\kappa}{\kappa-1} \leq p \leq 2$. 
Unfortunately the bound for the error term is much worse. 
\begin{lemma}\label{lemma:error:lp}
Let $\form(\vector{x}) \in \Z[\vector{x}]$ be a positive definite, non-singular, integral, homogeneous form satisfying \eqref{Birch_criterion}, and $\Gamma_{\form}$ be a set of regular values for the form $\form$. 
If $1 \leq p \leq 2$, then for all $\epsilon>0$, 
\begin{equation}\label{eq:error:lp:dyadic} 
\| \sup_{\Lambda \leq \lambda < 2\Lambda : \lambda \in \Gamma_\form} |E_\lambda| \|_{\ell^p(\Z^d) \to \ell^{p'}(\Z^d)} 
\lesssim_\epsilon 
\Lambda^{(1-\frac{d}{k})(\frac{2}{p}-1) - \gamma_\form(2-\frac{2}{p}) +\epsilon} 
\quad \text{for all} \quad \Lambda \geq 1.
\end{equation}
\end{lemma}

The proofs of these lemmas are motivated by proofs of Littman's theorem and its variants. 
Proofs of Littman's theorem often proceed by frequency decomposing the spherical average into pieces and finding $L^1 \to L^\infty$ and $L^2 \to L^2$ estimates with which to interpolate. 
In particular, we use a restricted weak-type argument that was used by Bourgain for the Euclidean spherical maximal function, by Ionescu \cite{Ionescu} for the discrete spherical maixmal function and also by Hu--Li for discrete restriction to the sphere in \cite{HuLi}. 

This is the same strategy that we follow in our proofs of these lemmas. 
In the next subsections we will deduce our lemmas. 

\subsection{Proof of Lemma~\ref{lemma:main:lp}}

Let $\kernel^{a/q}_\lambda$ denote the kernel (with domain $\Z^d$) associated to the convolution operator $M^{\form,a/q}_\lambda$. 
We start our proof by establishing an identity for these kernels. 
\begin{proposition}\label{proposition:main_term:l1tolinfty}
Let $\form(\vector{x}) \in \Z[\vector{x}]$ be a positive definite, non-singular, integral, homogeneous form satisfying \eqref{Birch_criterion}, and $\Gamma_{\form}$ be a set of regular values for the form $\form$. 
If $1 \leq a < q < \infty$ with $(a,q)=1$, then 
\begin{equation}\label{eq:kernel}
\kernel^{a/q}_\lambda(\vector{x}) 
= 
\eof{a\form(\vector{x})/q} \lambda^{-d/k} \widetilde{\dilate_{q\lambda^{-1/k}}\Psi}*d\sigma_{\form}(\lambda^{-1/k} \vector{x}) 
\quad \text{for} \quad \vector{x} \in \Z^d.
\end{equation}
\end{proposition}

\begin{proof}
Fix a form $\form$ satisfying the hypotheses of the proposition. 
We drop the dependence on $\form$ in our notation in order to simplify it. 

By Fourier inversion, our kernel is 
\begin{align*}
\kernel^{a/q}_\lambda(\vector{x}) 
& = 
\int_{\T^d} \eof{-\vector{x} \cdot \vector{\xi}} \sum_{\vector{m} \in \Z^d} G(a,q;\vector{m}) \Psi(q\vector{\xi}-\vector{m}) \widetilde{d\sigma}(\lambda^{1/k}[\vector{\xi}-\frac{\vector{m}}{q}]) \; d\vector{\xi}
\\ & = 
\sum_{\vector{m} \in \Z^d} G(a,q;\vector{m}) \int_{\T^d} \eof{-\vector{x} \cdot \vector{\xi}} \Psi(q\vector{\xi}-\vector{m}) \widetilde{d\sigma}(\lambda^{1/k}[\vector{\xi}-\frac{\vector{m}}{q}]) \; d\vector{\xi}
\\ & = 
\sum_{\vector{b} \in (\Z/q)^d} G(a,q;\vector{b}) \sum_{\vector{n} \in \Z^d} \int_{\T^d} \eof{-\vector{x} \cdot \vector{\xi}} \Psi(q\vector{\xi}-q\vector{n}-\vector{b}) \widetilde{d\sigma}(\lambda^{1/k}[\vector{\xi}-\vector{n}-\frac{\vector{b}}{q}]) \; d\vector{\xi}
\\ & = 
\sum_{\vector{b} \in (\Z/q)^d} G(a,q;\vector{b}) \int_{\R^d} \eof{-\vector{x} \cdot \vector{\xi}} \Psi(q\vector{\xi}-\vector{b}) \widetilde{d\sigma}(\lambda^{1/k}[\vector{\xi}-\frac{\vector{b}}{q}]) \; d\vector{\xi}
\end{align*}
The second equality follows since there is only a single term in the sum for each $\vector{\xi}$ while the third follows from writing every $\vector{m} \in \Z^d$ as $q\vector{n}+\vector{b}$ for some $\vector{n} \in \Z^d$ and a representative $\vector{b} \in \{0,1,\dots,q-1\}^d$ which we identify with $(\Z/q)^d$. 
Using the well-known translation, modulation and dilation symmetries of the (inverse) Fourier transform, we have 
\begin{align*}
\kernel^{a/q}_\lambda(\vector{x}) 
& = 
\sum_{\vector{b} \in (\Z/q)^d} G(a,q;\vector{b}) \int_{\R^d} \eof{-\vector{x} \cdot \vector{\xi}} \dilate_{q^{-1}}\Psi(\vector{\xi}-\frac{\vector{m}}{q}) \widetilde{\dilate_{\lambda^{-1/k}}d\sigma}(\vector{\xi}-\frac{\vector{m}}{q})
\\ & = 
\sum_{\vector{b} \in (\Z/q)^d} G(a,q;\vector{b}) \eof{-\frac{\vector{b} \cdot \vector{x}}{q}} \widetilde{\dilate_{q^{-1}\lambda^{1/k}}\Psi}*d\sigma(\lambda^{-1/k}\vector{x}).
\end{align*}
Identity \eqref{eq:kernel} immediately follows since 
\[
\sum_{\vector{b} \in (\Z/q)^d} G(a,q;\vector{b}) \eof{\frac{-\vector{b} \cdot \vector{x}}{q}}
=
\eof{\frac{a\form(\vector{x})}{q}}. 
\]
We leave this calculation to the reader since its just the inverse $(\Z/q)^d$-Fourier transform of the Gauss sum, and the Gauss sum is the $(\Z/q)^d$-Fourier transform of the function $\eof{\frac{a\form(\vector{x})}{q}}$.  
\end{proof}

Now that we know the structure of our kernel we will use a Littlewood--Paley decomposition and a circle method decomposition to arbitrage $\ell^1 \to \ell^\infty$ and $\ell^2 \to \ell^2$ estimates to deduce Lemma~\ref{lemma:main:lp}. 
In particular, the following lemma is motivated by the decompositions in \cite{Ionescu} when $k=2$ and \cite{Hughes_restricted} when $k \geq 3$. 
With this in mind, we introduce a low-high frequency decomposition in the analytic aspect. 

For $\Delta \in (0,1)$, define the low frequency piece by the Fourier multipler
\[
\widehat{M^{a/q,\low}_\lambda}(\vector{\xi}) 
:= \Psi(2q\Delta \lambda^{1/k}[\xi-\vector{m}/q]) \widehat{M^{a/q}_\lambda}(\vector{\xi}).
\]
The high frequency piece is defined as 
\[
M^{a/q,\high}_\lambda := M^{a/q}_\lambda - M^{a/q,\low}_\lambda.
\]
By the Fourier localization of $M^{a/q}_\lambda$ we have the restriction that $\Delta \gtrsim \lambda^{-1/k}$. 
\begin{lemma}\label{lemma:Birch:main_term}
Suppose that  $\Delta \in (0,1)$. 
Let $\form(\vector{x}) \in \Z[\vector{x}]$ be a positive definite, non-singular, integral, homogeneous form satisfying \eqref{Birch_criterion}, and $\Gamma_{\form}$ be a set of regular values for the form $\form$. 
If $1 \leq a < q \lesssim \Lambda^{1/k}$ with $(a,q)=1$, then each major arc piece $M^{a/q,\form} = M^{a/q,\form,\low} + M^{a/q,\form,\high}$ decomposes into a low frequency and high frequency piece such that 
\begin{equation}\label{eq:Birch:low}
\| \sup_{\Lambda \leq \lambda < 2\Lambda} | M^{a/q,\low}_\lambda| \|_{\ell^1(\Z^d) \to \ell^\infty(\Z^d)}
\lesssim 
(q\Delta)^{-1} \Lambda^{-\frac{d}{k}}
\end{equation}
and
\begin{equation}\label{eq:Birch:high}
\| \sup_{\Lambda \leq \lambda < 2\Lambda} |M^{a/q,\high}_\lambda| \|_{\ell^2(\Z^d) \to \ell^2(\Z^d)}
\lesssim_\epsilon
\Lambda^{\epsilon} q^{-\frac{3}{2}} \Delta^{\kappa-\frac{3}{2}}.
\end{equation}
\end{lemma}

\begin{proof}
The $\ell^2 \to \ell^2$-estimate is proved using Bourgain's $L^2$-estimates from \cite{Bourgain} as in \cite{Ionescu,Hughes_restricted}. 
We only sketch the proof of the $\ell^1 \to \ell^\infty$-estimate.

When $\form(\vector{x}) := \sum_{i=1}^d |x_i|^2$, then we have the following known bound for the continuous spherical measure
\begin{equation}\label{estimate:spherical}
\widetilde{\dilate_{t}\Psi}*d\sigma_{\form}(\vector{x}) 
\lesssim_d 
t^{-1} (1+|\vector{x}/t|)^{-2d}.
\end{equation}
See for instance page 1415 of \cite{Ionescu} where \eqref{estimate:spherical} is used to bound the discrete spherical maximal function, or (5.5.12) of \cite{Grafakos} for its derivation. 
The estimate \eqref{estimate:spherical} also holds for our varieties $\{ \vector{x} \in \R^d : \form(\vector{x}) = 1 \}$ since the proof only relies on the dimensionality of the measure $d\sigma_\form$. 
To be precise all we require is that 
\begin{equation}\label{eq:measure_dimension}
d\sigma(B_r) \lesssim \min\{ 1, r^{d-1} \}
\quad \text{for all balls of radius} \; r>0 \; \text{in} \; \R^d.
\end{equation}
This implies that for each $\vector{x} \in \Z^d$ we have 
\begin{align*}
\sup_{\Lambda \leq \lambda < 2\Lambda : \lambda \in \Gamma_\form} |M^{a/q}_\lambda f(\vector{x})|
& \lesssim_d 
|f| * [\sup_{\Lambda \leq \lambda < 2\Lambda : \lambda \in \Gamma_\form} (q\Delta)^{-1} \lambda^{-\frac{d}{k}} (1+|\lambda^{-1/k}\cdot|)^{-2d}] (\vector{x})
\\ & \lesssim_d 
|f| * [(q\Delta)^{-1} \Lambda^{-\frac{d}{k}} (1+|\Lambda^{-1/k}\cdot|)^{-2d}] (\vector{x})
\\ & \lesssim_d 
\| f \|_{\ell^1(\Z^d)} \cdot \| (q\Delta)^{-1} \Lambda^{-\frac{d}{k}} (1+|\Lambda^{-1/k}\cdot|)^{-2d}] \|_{\ell^\infty(\Z^d)}
\\ & \lesssim_d 
(q\Delta)^{-1} \Lambda^{-\frac{d}{k}} \cdot \| f \|_{\ell^1(\Z^d)} 
\end{align*}
\end{proof}

\begin{remark}
Estimates \eqref{eq:Birch:surface_measure} and \eqref{eq:measure_dimension} and a Euclidean version of the low--high decomposition suffice to prove Euclidean versions of Theorem~\ref{theorem:discrete_improving:Birch} generalizing Littman's theorem. 
\end{remark}

We now return to the proof of Lemma~\ref{lemma:main:lp}. 
Fix $\Lambda \geq 1$. 
Let $X$ be a fixed, finite subset of $\Z^d$. 
Our first observation is that estimate \eqref{eq:Birch:low} implies that for $\Lambda \leq \lambda < 2\Lambda$. 
\begin{equation}\label{eq:Birch:main:l1tolinfty}
 \sup_{\Lambda \leq \lambda < 2\Lambda} |M_\lambda \indicator{X}(\vector{x})| 
\lesssim 
\Lambda^{\frac{2-d}{k}} |X|
\quad \text{for all} \quad \vector{x} \in \Z^d
\end{equation}
by taking $\Delta \eqsim \Lambda^{-1/k}$ and summing over the moduli $1 \leq q \lesssim \Lambda^{1/k}$. 
The implicit bound is independent of the set $X$. 
Consquently we are reduced to proving 
\begin{equation}\label{eq:Birch:reduction}
|\{ \vector{x} \in \Z^d : \sup_{\Lambda \leq \lambda < 2\Lambda} |M_\lambda \indicator{X}(\vector{x})| > T \}| 
\lesssim_\epsilon  
\Lambda^{\frac{d(2-\kappa)}{k}+\epsilon}T^{-\kappa}|X|^{\kappa-1}
\quad \text{for} \quad
0 < T \lesssim \Lambda^{\frac{2-d}{k}} |X|.
\end{equation}

Since 
\begin{align*}
| \{ \vector{x} \in \Z^d : \sup_{\Lambda \leq \lambda < 2\Lambda} |M_\lambda f(\vector{x})| > T \} | 
& \leq 
| \{ \vector{x} \in \Z^d : \sup_{\Lambda \leq \lambda < 2\Lambda} |\sum_{q \leq Q} M^{q,\low}_\Lambda f(\vector{x})| > T/3 \} | 
\\ & \quad \quad + | \{ \vector{x} \in \Z^d : \sup_{\Lambda \leq \lambda < 2\Lambda} |\sum_{q \leq Q} M^{q,\high}_\Lambda f(\vector{x})| > T/3 \} | 
\\ & \quad \quad + | \{ \vector{x} \in \Z^d : \sup_{\Lambda \leq \lambda < 2\Lambda} |\sum_{q > Q} M^{q}_\Lambda f(\vector{x})| > T/3 \} |,
\end{align*}
the inequalities of Lemma~\ref{lemma:Birch:main_term} combine to imply that 
\begin{equation}\label{eq:Birch:maximal_restricted_estimate}
| \{ \vector{x} \in \Z^d : \sup_{\Lambda \leq \lambda < 2\Lambda} |M_\Lambda f(\vector{x})| > T \} | 
\lesssim_\epsilon 
\Lambda^{\epsilon} [Q^{1/2}\Delta^{\kappa-\frac{3}{2}} + Q^{2-\kappa}]^{2} T^{-2}|X|
\end{equation}
provided that we choose $Q$ and $\Delta$ such that
\(
Q \Delta ^{-1}\Lambda^{-d/k} |X|
\lesssim 
T.
\)

The inequality \eqref{eq:Birch:maximal_restricted_estimate} shows that we want to choose $Q^{1/2}\Delta^{\kappa-\frac{3}{2}} \eqsim Q^{2-\kappa}$ which is $\Delta \eqsim Q^{-1}$; we now make this assumption. 
Our restriction then takes the form
\(
Q^2 \Lambda^{-d/k} |X| 
\lesssim 
T
\) 
which is consistent with our reduction \eqref{eq:Birch:reduction}. 
Choosing $Q^2 \eqsim \Lambda^{d/k}T|X|^{-1}$, we deduce that 
\begin{align*}
| \{ \vector{x} \in \Z^d : \sup_{\Lambda \leq \lambda < 2\Lambda} |M_\lambda f(\vector{x})| > T \} | 
& \lesssim_\epsilon
\Lambda^{\epsilon} [\Lambda^{d/k}T|X|^{-1}]^{2-\kappa}T^{-2}|X|
\\ & \lesssim_\epsilon 
\Lambda^{\frac{d(2-\kappa)}{k}+\epsilon}T^{-\kappa}|X|^{\kappa-1}
\end{align*}
for all $0 < T \leq 1$ as desired. 
\subsection{Proof of Lemma~\ref{lemma:error:lp}}

The proof of Lemma~\ref{lemma:error:lp} follows by interpolating the error term estimate \eqref{eq:error_bound} in Magyar's theorem with estimate \eqref{eq:Birch:error:dyadic} below.

\begin{proposition}
We have the following $\ell^1 \to \ell^{\infty}$-improving estimate for the error term:
\begin{equation}\label{eq:Birch:error:dyadic}
\| \sup_{\Lambda \leq \lambda < 2\Lambda : \lambda \in \Gamma_\form} |E_\lambda f| \|_{\ell^\infty(\Z^d)}
\lesssim 
\Lambda^{1-\frac{d}{k}} \| f \|_{\ell^1(\Z^d)}. 
\end{equation}
\end{proposition}

\begin{proof}
Since $E_\lambda = A_\lambda - M_\lambda$, the trivial bound \eqref{eq:Birch:trivial_bound} for the dyadic maximal function and our bound \eqref{eq:Birch:main:l1tolinfty} for the main term imply \eqref{eq:Birch:error:dyadic}.
%
%
\end{proof}

\section{The discrete spherical averages}\label{section:sphere}

In this section we refine our main term analysis from Section~\ref{section:BirchMagyar} in order to prove Theorem~\ref{theorem:discrete_improving}. 
We now recall a decomposition of Magyar which, for the discrete spherical averages, improves upon the error term bound \eqref{eq:error_bound} in Magyar~\ref{Magyar:Birch}. 
\begin{Magyar}[\cite{Magyar:distribution}]
Suppose that $\form(\vector{x}) = x_1^2 + \cdots + x_d^2$ for $d \geq 4$.
For each $\lambda \in \N$ (with $\lambda$ assumed to be odd when $d=4$) and all $\xi \in \T^d$, we have $A^\form_\lambda = M^\form_\lambda+E^\form_\lambda$ for $\lambda \in \N$ 
such that 
\begin{equation}\label{eq:error_bound:sphere}
\|\widehat{E_\lambda}\|_{L^\infty(\T^d)} 
\lesssim_\epsilon 
\lambda^{\epsilon-\frac{d-3}{4}}
\quad \text{for all} \quad \epsilon>0.
\end{equation}
The main term $M^\form_\lambda$ decomposes into a sum of pieces $M^\form_\lambda = \sum_{q=1}^{\Lambda^{1/2}} M^{\form,q}_\lambda$ where $\Lambda$ is the smallest dyadic integer (that is of the form $2^j$ for some $j \in \N$) such that $\Lambda \leq \lambda < 2\Lambda$ and $M^\form_\lambda$ has Fourier multiplier
\begin{align*}
\widehat{M_\lambda^{q}}(\vector{\xi})
:= 
\sum_{\vector{m} \in \Z^d}  \Kloosterman(q,\lambda;\vector{m}) \Psi(2q\sqrt{\lambda}[\vector{\xi}-\vector{m}/q]) \widetilde{d\sigma}(\sqrt{\lambda}[\vector{\xi}-\vector{m}/q]).
\end{align*}
Each piece of the main term satifies the following ``Weil bound'': 
\begin{equation}\label{eq:Weil}
\|\widehat{M_\lambda^q}\|_{L^\infty(\T^d)} 
\lesssim_\epsilon
q^{-(\frac{d-1}{2})+\epsilon} (q,\lambda)^{1/2}
\quad \text{for all} \quad \epsilon>0.
\end{equation}
\end{Magyar}

In \eqref{eq:Weil} and below $(a,b)$ represents the greatest common divisor of two integers $a,b$, and 
$\Kloosterman(q,\lambda;\vector{m})$ is the Kloosterman/Salie sum defined as 
\[
\Kloosterman(q,\lambda;\vector{m}) 
:= 
q^{-d} \sum_{a \in (\Z/q)^\times} \eof{-\frac{a\lambda}{q}} \sum_{\vector{b} \in (\Z/q)^d} \eof{\frac{a\form(\vector{b})+\vector{b}\cdot\vector{m}}{q}}
\]
for $q \in \N$ and $\vector{m} \in \Z^d$. 
For the estimate \eqref{eq:error_bound:sphere}, see (1.9) of Lemma~1 in \cite{Magyar:distribution}. 
When comparing, note that we have normalized our surface measure to be a probabilty measure. 
The estimate \eqref{eq:Weil} follows from the famous Weil bounds for Kloosterman sums:
\begin{equation*}\label{Weil_bound}
\Kloosterman(q,\lambda;\vector{m}) 
\lesssim_\epsilon
q^{-(\frac{d-1}{2})+\epsilon} (q,\lambda)^{1/2} 
\quad \text{for each} \quad \vector{m} \in \Z^d 
\quad \text{and for all} \quad \epsilon>0. 
\end{equation*}
Also note that $M^q_\lambda = \sum_{a \in (\Z/q)^*} M^{a/q}_\lambda$. 

Our strategy is the same as before; we need to prove bounds for the main term and the error term. 
We first improve our bounds for the main term. 
\begin{lemma}\label{lemma:sphere:main}
For $\form(\vector{x}) := \sum_{i=1}^d x_i^2$, $d \geq 4$ and $p = \frac{d+1}{d-1}$, 
\begin{equation}
\| M_\lambda\|_{\ell^{p,1}(\Z^d) \to \ell^{p',\infty}(\Z^d)} 
\lesssim_\epsilon
\lambda^{\epsilon-\frac{d}{2}(\frac{2}{p}-1)}
\quad \text{for all} \quad \epsilon>0.
\end{equation}
\end{lemma}

Subsequently, we improve our bound for the error term. 
\begin{lemma}\label{lemma:sphere:error}
When $\form(\vector{x}) := \sum_{i=1}^d x_i^2$, $d \geq 4$ and $1 \leq p \leq 2$. 
\begin{equation}
\| E_\lambda \|_{\ell^p(\Z^d) \to \ell^{p'}(\Z^d)} 
\lesssim_\epsilon
\lambda^{-\frac{d}{2}(\frac{2}{p}-1) + (\frac{2}{p}-1)-(\frac{d-3}{4})(2-\frac{2}{p}) +\epsilon}
\quad \text{for all} \quad \epsilon>0.
\end{equation}
\end{lemma}

A simple computation reveals that $(\frac{2}{p}-1)-(\frac{d-3}{4})(2-\frac{2}{p}) < 0$ for $p > \frac{d+1}{d-1}$. 
Theorem~\ref{theorem:discrete_improving} follows immediately from combining Lemmas~\ref{lemma:sphere:main} and \ref{lemma:sphere:error}.

\subsection{Proof of Lemma~\ref{lemma:sphere:main}}

We use the low-high decomposition of Section~\ref{section:BirchMagyar}.
First we have 
\[
\| \sum_{Q<q \lesssim \lambda^{1/2}} M^{q}_\lambda \|_{\ell^2(\Z^d) \to \ell^2(\Z^d)} 
\lesssim_{\epsilon}
Q^{\epsilon-\frac{d-3}{2}} \lambda^{\epsilon}
\]
because summing over $(q,\lambda)^{1/2}$ only contributes a factor of $\lambda^{\epsilon}$ on average. 
Fourier transform estimates for the spherical measure and the Weil bound for Kloosterman sums \eqref{eq:Weil} implies the following bound on the high frequency pieces
\begin{align*}
\| M^{q,\high}_\lambda \|_{\ell^2(\Z^d) \to \ell^2(\Z^d)}
& \lesssim_\epsilon
q^{\epsilon-\frac{d-1}{2}} (q,\lambda)^{1/2} (q\Delta)^{-\frac{d-1}{2}}
= 
q^{\epsilon} (q,\lambda)^{1/2} \Delta^{-\frac{d-1}{2}}.
\end{align*}
Finally the low frequency bound is the same as before: 
\begin{align*}
\| M^{a/q,\low}_\lambda \|_{\ell^1(\Z^d) \to \ell^\infty(\Z^d)}
& \lesssim
(q\Delta)^{-1} \lambda^{-\frac{d}{2}}.
\end{align*}

The proof now proceeds by the restricted weak-type argument of Section~\ref{section:BirchMagyar}. 
We will be brief in our description of this. 
Let $X$ be a fixed subset of $\Z^d$. 
By the same reduction as in the proof of Lemma~\ref{lemma:main:lp}, we seek to prove
\begin{equation*}
|\{ |M_\lambda \indicator{X}(\vector{x})| > T \}| 
\lesssim_\epsilon 
\lambda^{-\frac{d(d-3)}{4}+\epsilon} T^{-\frac{d+1}{2}} |X|^{\frac{d-1}{2}} 
\quad \text{for} \quad 
0 < T \lesssim \lambda^{1-\frac{d}{2}} |X|.
\end{equation*}
Combining the above estimates for low pieces $M_\lambda^{q,\low}$ and high pieces $M_\lambda^{q,\high}$ implies that 
\begin{align*}
| \{ \vector{x} \in \Z^d : M_\Lambda f(\vector{x}) > T \} | 
& \lesssim_\epsilon 
\lambda^{\epsilon} \left( [Q\Delta^{\frac{d-1}{2}}]^{2}  + Q^{3-d} \right) T^{-2}|X|
\end{align*}
provided that we choose $Q$ and $\Delta$ such that \( Q \Delta ^{-1}\lambda^{-d/2} |X| \lesssim T\). 
We choose $\Delta \eqsim Q^{-1}$ so that \( Q^2 \lambda^{-d/2} |X| \eqsim T \). 
Plugging this in we ascertain that 
\begin{align*}
| \{ \vector{x} \in \Z^d : M_\lambda f(\vector{x}) > T \} | 
& \lesssim_\epsilon
\lambda^{\epsilon} [\lambda^{d/2}T|X|^{-1}]^{-\frac{d-3}{2}} \cdot T^{-2}|X|
\\ & =
\lambda^{-\frac{d(d-3)}{4}+\epsilon} T^{-\frac{d+1}{2}} |X|^{\frac{d-1}{2}}
\end{align*}
for all $0<T\leq1$ as desired.
%

\subsection{Proof of Lemma~\ref{lemma:sphere:error}}

We have Magyar's Kloosterman bound \eqref{eq:error_bound:sphere} which says that we may take $\gamma_{\form} := \frac{d-3}{4}$ for $\form(\vector{x}) := \sum_{i=1}^d |x_i|^2$ with $d \geq 4$. 
Interpolating with the $\ell^1(\Z^d) \to \ell^\infty(\Z^d)$-bound \eqref{eq:Birch:error:dyadic}, we obtain the lemma.

\section{The dyadic discrete $k$-spherical maximal functions}\label{section:sphere:dyadic}

Let $\form(\vector{x}) := \sum_{i=1}^d |x_i|^k$ for integers $k \geq 2$, and $\Gamma_\form$ a set of regular values for $\form$. 
(For these forms, $\Gamma_\form$ contains all sufficiently large natural numbers when $d \gtrsim k^2$.) 
For degrees $k \geq 3$, define the dimensions
\[
d_k 
:= 
k^2 - \max_{2 \leq j \leq k-1} \left\{ \frac{k j - \min(2^j+2,j^2+j)}{j-j+1} \right\},
\]
and the exponents 
\[
p_{d,k} 
:= 
\max\{ \frac{d}{d-k}, 1+\frac{1}{2\gamma_{d,k}} \}
\quad \text{for} \quad k \geq 3 \quad \text{and} \quad d > d_k
\]
where $\gamma_{d,k}$ is defined as 
\[
k \cdot \gamma_{d,k}
:= 
\begin{cases} 
(d-d_k)(k^2+k-d_k)^{-1} & \text{if } d_k \leq d \leq k^2+k, 
\\ 1 + (d-k^2-k)(\max\{ 2^{1-k},(k^2-k)^{-1} \} ) &\text{if } d > k^2+k. 
\end{cases} 
\]
When $k=2$, define $d_2 := 4$, $\gamma_{d,2} := 1-\frac{d}{4}$ and $p_{d,2} := \frac{d}{d-2}$. 
These bounds relied on the works \cite{Bourgain:Vinogradov,BDG,BrudernRobert,Wooley:asymptotic}. 
%

For $k$-spheres we may generalize Theorem~\ref{theorem:improving:dyadic} to the following. 
\begin{theorem}\label{theorem:kspheres}
For $k \geq 2$, $d \geq d_k$ and $p > p_{d,k}$, 
\begin{equation}
\|\sup_{\Lambda \leq \lambda < 2\Lambda} |A_\lambda|\|_{\ell^p(\Z^d) \to \ell^{p'}(\Z^d)} 
\lesssim
\Lambda^{-\frac{d}{k}(\frac{2}{p}-1)}.
\end{equation}
\end{theorem}

In the proof of this theorem we will see that we may easily deduce the restricted weak-type bound claimed in Theorem~\ref{theorem:improving:dyadic}. 
As before we will break into two lemmas which handle the main term and error term respectively, and from which Theorem~\ref{theorem:kspheres} follows immediately. 

\begin{lemma}\label{lemma:kspheres:main}
For $k \geq 2$, $d \geq 2k+1$ and $p := \frac{d}{d-k}$, 
\begin{equation}
\|\sup_{\Lambda \leq \lambda < 2\Lambda} |M_\lambda|\|_{\ell^{p,1}(\Z^d) \to \ell^{p',\infty}(\Z^d)} 
\lesssim
\Lambda^{-\frac{d}{k}(\frac{2}{p}-1)}.
\end{equation}
\end{lemma}

\begin{lemma}\label{lemma:kspheres:error}
For $k \geq 2$, $d \geq d_k$ and $p > p_{d,k}$, 
\begin{equation}
\|\sup_{\Lambda \leq \lambda < 2\Lambda} |E_\lambda|\|_{\ell^p(\Z^d) \to \ell^{p'}(\Z^d)} 
\lesssim_\epsilon
\Lambda^{-\frac{d}{k}(\frac{2}{p}-1) + (\frac{2}{p}-1)-\gamma_{d,k}(2-\frac{2}{p}) +\epsilon}.
\end{equation}
When $k=2$ one may remove the $\epsilon$-loss.
\end{lemma}

\subsection{Proof of Lemma~\ref{lemma:kspheres:main}}

We first handle the main term which relies on \emph{Steckin's estimate} (see \cite{Hughes_restricted}): 
\begin{equation}\label{eq:Steckin}
G_\form(a,q;\vector{m}) 
\lesssim_{k}
q^{-d/k}
\end{equation}
where the bound is uniform in $a$ coprime to $q$ and $\vector{m} \in \Z^d$, 
and Bruna--Nagel--Wainger's estimates - see \cite{BNW} - imply that 
\[
|\widetilde{d\sigma_{\form}}(\vector{\xi})| \lesssim (1+|\vector{\xi}|)^{-\frac{d-1}{k}}. 
\]
The effect of the following is that we may replace $\kappa_\form$ in Section~\ref{section:BirchMagyar} with $d/k$. 
An essential point below is that we have no $\epsilon$-loss; that is, no extraneous factors of $\Lambda^{\epsilon}$ like before. 

We first sharpen the bound for our main term. 
As shown in \cite{MSW} and \cite{Hughes_restricted}, the Gauss bound for Gauss sums and Steckin's estimate \eqref{eq:Steckin} implies that 
\begin{align*}
\| \sup_{\Lambda \leq \lambda < 2\Lambda : \lambda \in \Gamma_\form} |M^{a/q}_\lambda f| \|_{\ell^2(\Z^d)} 
& \lesssim 
q^{-\frac{d}{k}} \|f\|_{\ell^2(\Z^d)}.
\end{align*}
Consequently, 
\begin{align*}
\| \sup_{\Lambda \leq \lambda < 2\Lambda : \lambda \in \Gamma_\form} |\sum_{Q < q \lesssim \Lambda^{1/2}} \sum_{a \in (\Z/q)^*} M^{a/q}_\lambda f| \|_{\ell^2(\Z^d)} 
& \lesssim 
Q^{2-\frac{d}{k}} \|f\|_{\ell^2(\Z^d)}.
\end{align*}

Using the low-high decomposition of Section~\ref{section:BirchMagyar}, we have still have \eqref{eq:Birch:low} for the low-frequency piece and for the high-frequency piece we now have 
\begin{equation*}
\| \sup_{\Lambda \leq \lambda < 2\Lambda} |M^{a/q,\high}_\lambda| \|_{\ell^2(\Z^d) \to \ell^2(\Z^d)}
\lesssim
q^{-d/k} (q\Delta)^{\frac{d-1}{k}-\frac{1}{2}}
= 
q^{-(\frac{1}{k}+\frac{1}{2})} \Delta^{\frac{d-1}{k}-\frac{1}{2}}.
\end{equation*}
%
%
Running the restricted weak-type argument as before we conclude the lemma.

\subsection{Proof of Lemma~\ref{lemma:kspheres:error}}

We now handle the error term. 
From \cite{MSW} and \cite{ACHK}, we have that 
\begin{equation}\label{eq:sphere:error:dyadic}
\|\sup_{\Lambda \leq \lambda < 2\Lambda} |E_\lambda f|\|_{\ell^2} 
\lesssim_\epsilon
\Lambda^{\epsilon-\gamma_{d,k}} \|f\|_{\ell^2}
\end{equation}
where one may remove the $\epsilon$ in \eqref{eq:sphere:error:dyadic} when $k=2$. 
Interpolating with the $\ell^1(\Z^d)$-bound from Section~\ref{section:BirchMagyar} (see the proof of Lemma~\ref{lemma:error:lp}), we obtain the lemma. 


%

\section{The discrete spherical maximal function in four dimensions}\label{section:4d}

In this section let $\form(\vector{x}) := x_1^2+x_2^2+x_3^2+x_4^2$. 
We have the following $\ell^p(\Z^d)$-improving estimates for Magyar's discrete dyadic spherical maximal functions. 
\begin{theorem}\label{theorem:4d:dyadic}
If $d = 4$, then for each $q>2$ and $1 \leq p < q$ there exists $\delta_{p,q} > 0$ such that 
\begin{equation}\label{eq:4d:dyadic}
\| \sup_{\Lambda \leq \lambda < \Lambda} |A_\lambda f|  \|_{\ell^{q}(\Z^d)} 
\lesssim 
\Lambda^{-\delta_{p,q}} \|f\|_{\ell^p(\Z^d)}.
\end{equation}
We assume that the $\lambda$ in the supremum are restricted to be odd. 
\end{theorem}
\noindent 
Theorem~\ref{theorem:4d} follows from summing up Theorem~\ref{theorem:4d:dyadic} over $\Lambda = 2^j$ for integers $j \geq 0$. 

\begin{proof}[Proof of Theorem~\ref{theorem:4d:dyadic}]
Our trivial bound \eqref{eq:trivial_bound} says that 
\[
\|\sup_{\Lambda \leq < \lambda < 2\Lambda} |A_\lambda f|\|_{\ell^\infty(\Z^4)} 
\lesssim 
\Lambda^{-1} \|f\|_{\ell^1(\Z^4)}.
\]
We need one additional ingredient which is from the author's thesis.
\begin{theorem}[\cite{Hughes_thesis}]\label{theorem:l2}
For all $\lambda>0$, we have 
\[
\| \sup_{\Lambda \leq \lambda < 2\Lambda} |A_\lambda f| \|_{\ell^2(\Z^4)} 
\lesssim (\log \Lambda)^2 \|f\|_{\ell^2(\Z^4)}.
\]
The implicit constant is independent of $\Lambda$.
\end{theorem}

Interpolating our trivial bound and Theorem~\ref{theorem:l2} we obtain that 
\begin{equation}\label{eq:dyadic:weak}
\| \sup_{\Lambda \leq \lambda < \Lambda} |A_\lambda f|  \|_{\ell^{p'}(\Z^4)} 
\lesssim
\Lambda^{-(\frac{2}{p}-1)} (\log \Lambda)^2 \|f\|_{\ell^p(\Z^4)}
\quad \text{for} \quad 1 \leq p \leq 2.
\end{equation}
The nesting property of $\ell^p$-spaces implies that 
\[
\| \sup_{\Lambda \leq \lambda < 2\Lambda} |A_\lambda f| \|_{\ell^{q}(\Z^4)} 
\lesssim
\Lambda^{-(\frac{2}{q'}-1)} (\log \Lambda)^2 \|f\|_{\ell^p(\Z^4)}
\quad \text{for all} \quad 
q \geq 2
\quad \text{and} \quad 
1 \leq p \leq q'
\]
which implies Theorem~\ref{theorem:4d:dyadic} in this range. 
Interpolating \eqref{eq:dyadic:weak} with the trivial $\ell^\infty$-bound
\[
\| \sup_{\Lambda \leq \lambda < 2\Lambda} |A_\lambda f| \|_{\ell^\infty(\Z^4)} 
\leq 
\|f\|_{\ell^\infty(\Z^4)},
\]
we obtain for each $q \geq 2$ that 
\[
\| \sup_{\Lambda \leq \lambda < 2\Lambda} |A_\lambda f| \|_{\ell^{s}(\Z^4)} 
\lesssim
[\Lambda^{-(\frac{2}{p}-1)} (\log \Lambda)^2]^{q'/r} \|f\|_{\ell^{r}(\Z^4)}
\quad \text{for all} \quad 
r \geq q
\quad \text{and} \quad 
s = \frac{q'}{rq}.
\]
Consequently, \eqref{eq:4d:dyadic} holds for all $q \geq 2$ and $p \geq 2$. 
\end{proof}

\section{Further questions}\label{section:conclusion}

Let $X$ be a finite subset of $\Z^d$ and $\mu_X := |X|^{-1} {\bf 1}_X$ be its normalized probability measure. 
By Young's inequality, we have 
\begin{equation}\label{equation:Young:general}
\|f*\mu_X\|_{\ell^{p'}(\Z^d)}
\lesssim 
|X|^{1-\frac{2}{p}} \|f\|_{\ell^p(\Z^d)}.
\end{equation}
For instance if $X = X_r$ is taken to be the integer points of a large ball $\{ \vector{x} \in \Z^d : |\vector{x}| < r \}$ for $r>0$, then 
\[
\|f*\mu_X\|_{\ell^{p'}(\Z^d)}
\lesssim 
r^{d(1-\frac{2}{p})} \|f\|_{\ell^p(\Z^d)}
\]
since $\# \{ \vector{x} \in \Z^d : |\vector{x}| < r \} \eqsim r^d$ for all large $r>0$. 
Moreover testing against Example~\ref{example:bigball} (big ball example) we see that this bound is sharp (aside possibly from the implicit constants). 
Therefore \eqref{equation:Young:general} is in general sharp for all $1 \leq p \leq 2$, and another way to interpret Theorem~\ref{theorem:discrete_improving} is that the discrete sphere smooths almost as well as the discrete ball which contains the same number of points as the sphere in the range $\frac{d+1}{d-1} \leq p < 2$. 
This begs with a few questions. 
\begin{question}[Michael Fryers]
For a random subset $X$ what is the best constant in the inequality \eqref{equation:Young:general}?
\end{question}

\begin{question}[Michael Fryers]
Is \eqref{equation:Young:general} still true if we take $X$ to be a perturbation of the discrete spherical measures $\sigma_\lambda$?
\end{question}

\begin{question}
What are the extremizers satisfying \eqref{equation:Young:general} for all $1 \leq p \leq 2$? 
Do they resemble the ball?
\end{question}
%

\begin{appendix}
\section{$\ell^p(\Z^d)$-bounedness for Magyar's theorem}

Theorem~4 of \cite{Magyar:ergodic} is only stated for $\ell^2(\Z^d)$. 
Since it has come up in conversation on multiple occasions, we record the range of $\ell^p(\Z^d)$-boundedness to which Magyar's theorem extends and briefly indicate how. 
Standard density arguments extend the range of Magyar's pointwise ergodic theorem - Theorem~3 in \cite{Magyar:ergodic} - to the same range of $L^p$-spaces as $\ell^p(\Z^d)$-spaces below. 

In this appendix assume that $\form(\vector{x}) \in \Z[\vector{x}]$ is a non-singular, integral, homogeneous form with a set of regular values $\Gamma_\form$ satisfying \eqref{Birch_criterion}, define the maximal function for functions $f : \Z^d \to \C$ by
\[
A_* f(\vector{x}) 
:= 
\sup_{\lambda \in \Gamma_\form} |A_\lambda f(\vector{x})| 
\quad \text{for each} \quad \vector{x} \in \Z^d. 
\]
\begin{Magyar}
The maximal function $A_*$ is bounded on $\ell^p(\Z^d)$ for $p > \max\{ \frac{\kappa}{\kappa-1}, \frac{2(\gamma+1)}{2\gamma+1} \}$ where $\kappa = \kappa_\form$ and $\gamma = \gamma_\form$ are defined by \eqref{eq:Birch:parameters}.
\end{Magyar}

As in Section~\ref{section:BirchMagyar}, we have $A_\lambda = M_\lambda+E_\lambda$. 
This implies that for each function $f: \Z^d \to \C$ and each $\vector{x} \in \Z^d$
\[
A_* f(\vector{x}) 
\leq 
( \sum_{q=1}^\infty \sum_{a \in (\Z/q)^*} |M_*^{a/q} f(\vector{x})| ) 
+ E_*f(\vector{x})
\]
where 
\[
M_*^{a/q}f(\vector{x}) 
:= 
\sup_{\lambda \in \Gamma_\form} |M_\lambda^{a/q} f(\vector{x})|
\]

and 
\[
E_*f(\vector{x}) 
:= 
\sup_{\lambda \in \Gamma_\form} |E_\lambda f(\vector{x})|.
\] 
Our strategy is then to give sufficiently decaying $\ell^p$-bounds bounds on $M_*^{a/q}$ to sum over $a$ and $q$ and to prove $\ell^p$-boundedness of $E_*$ through its dyadic counterparts. 

\begin{lemma}\label{lemma:Magyar:main_piece}
If $\form(\vector{x}) \in \Z[\vector{x}]$ is a non-singular, integral, homogeneous form satisfying \eqref{Birch_criterion}, then for $1 \leq p \leq 2$ we have the bound 
\begin{equation}
\| M_*^{a/q} \|_{\ell^p(Z^d) \to \ell^p(\Z^d)} 
\lesssim_\epsilon 
q^{\epsilon-\kappa(2-\frac{2}{p})}
\quad \text{for each} \quad \epsilon>0.
\end{equation}
\end{lemma}
\noindent 
Summing over $a$ and $q$ immediately yields the following corollary. 
We see that we only need $2-\kappa(2-\frac{2}{p})<0$ which is when $p > \frac{\kappa}{\kappa-1}$. 
\begin{corollary}
Let $\form(\vector{x}) \in \Z[\vector{x}]$ be a non-singular, integral, homogeneous form satisfying \eqref{Birch_criterion}, and $\Gamma_{\form}$ be a set of regular values for the form $\form$. 
If $\frac{\kappa}{\kappa-1} < p \leq 2$, then 
\begin{equation}\label{eq:main_term} 
\| \sup_{\lambda \in \Gamma_\form} |M_\lambda f| \|_{\ell^{p}(\Z^d)} 
\lesssim_{\form,p} 
\| f \|_{\ell^{p}(\Z^d)}.
\end{equation}
\end{corollary}

Our lemma for the error term is the following. 
\begin{lemma}\label{lemma:Magyar:error}
Let $\form(\vector{x}) \in \Z[\vector{x}]$ be a non-singular, integral, homogeneous form satisfying \eqref{Birch_criterion}, and $\Gamma_{\form}$ be a set of regular values for the form $\form$. 
If $1 \leq p \leq 2$, then for all $\epsilon>0$, 
\begin{equation}\label{eq:error:lp:dyadic} 
\| \sup_{\Lambda \leq \lambda < 2\Lambda : \lambda \in \Gamma_\form} |E_\lambda f| \|_{\ell^p(\Z^d)}
\lesssim_\epsilon 
\Lambda^{(\frac{2}{p}-1) - \gamma_\form(2-\frac{2}{p}) +\epsilon} 
\|f\|_{\ell^{p}(\Z^d)} 
\quad \text{for all} \quad \Lambda \geq 1.
\end{equation}
\end{lemma}

Summing \eqref{eq:error:lp:dyadic} over $\Lambda = 2^j$ for integers $j \geq 0$ immediately yields the following corollary.
\begin{corollary}
The error term maximal function $E_*$ is bounded on $\ell^p(\Z^d)$ for $p > \frac{2(\gamma+1)}{2\gamma+1} = 1+\frac{1}{2\gamma+1}$. 
\end{corollary}

\subsection{Proof of Lemma~\ref{lemma:Magyar:main_piece}}

We use the Magyar--Stein--Wainger approach of \cite{MSW}. 
For this we will need the following estimate for analogous Euclidean maximal functions. 
\begin{proposition}
If $\form(\vector{x}) \in \Z[\vector{x}]$ is a non-singular, integral, homogeneous form satisfying \eqref{Birch_criterion}, then the maximal function $\sup_{t>0} |F*d\sigma_\form|$, defined a priori for smooth compactly supported functions $F : \R^d \to \C$, is bounded on $L^p(\R^d)$ for $p>1+(2\kappa-2)^{-1} = \frac{2\kappa-1}{2(\kappa-1)}$. 
\end{proposition}

\begin{proof}
The proposition follows from Magyar's estimate \eqref{eq:Birch:surface_measure} by applying Theorem~A  of \cite{Rubio} since $\kappa>2$. 
\end{proof}

We resume the proof of Lemma~\ref{lemma:Magyar:main_piece}.
By \eqref{eq:kernel} of Proposition~\ref{proposition:main_term:l1tolinfty} we have that 
\[
\| M_*^{a/q} \|_{\ell^1(Z^d) \to \ell^1(\Z^d)} 
\lesssim 
1. 
\]
The Magyar--Stein--Wainger transference principle (Proposition~2.1 of \cite{MSW}) and the estimate \eqref{eq:Birch:Weyl_sum} imply that 
\[
\| M_*^{a/q} \|_{\ell^2(Z^d) \to \ell^2(\Z^d)} 
\lesssim 
q^{\epsilon-\kappa}.
\]
Interpolation handles the remaining $1<p<2$.

\subsection{Proof of Lemma~\ref{lemma:Magyar:error}}

The $\ell^2 \to \ell^2$-estimate is \eqref{eq:error_bound}. 
Interpolation with the $\ell^1 \to \ell^1$-estimate 
\[
\| \sup_{\Lambda \leq \lambda < 2\Lambda} |E_\lambda f| \|_{\ell^1(Z^d)} 
\lesssim 
\Lambda \| f \|_{\ell^1(Z^d)}
\]
concludes the lemma. 

To prove this $\ell^1 \to \ell^1$-estimate we note that the union bound implies that 
\[
\| \sup_{\Lambda \leq \lambda < 2\Lambda} |A_\lambda f| \|_{\ell^1(Z^d)} 
\leq 
\Lambda \| f \|_{\ell^1(Z^d)}
\]
since each average is a contraction on $\ell^1$. 
Meanwhile, Proposition~\ref{proposition:main_term:l1tolinfty} implies that
\begin{equation}
\| \sup_{\Lambda \leq \lambda < 2\Lambda} |M_\lambda| \|_{\ell^1(Z^d) \to \ell^1(\Z^d)} 
\lesssim_\epsilon 
\Lambda^{2/k}.
\end{equation}
Since $k \geq 2$, we have that 
\[
\| \sup_{\Lambda \leq \lambda < 2\Lambda} |E_\lambda f| \|_{\ell^1(Z^d)} 
= 
\| \sup_{\Lambda \leq \lambda < 2\Lambda} |A_\lambda f - M_\lambda f| \|_{\ell^1(Z^d)} 
\lesssim 
\Lambda \| f \|_{\ell^1(Z^d)}.
\]

\end{appendix}

%
%

\bibliographystyle{amsalpha}

\end{document}